\documentclass[a4paper,12pt]{article}
\usepackage{Vnichtdiffbar}

\begin{document}

 \title{Note on the (non-)smoothness of discrete time value functions}
 \author{S\"oren Christensen\thanks{Mathematisches Seminar, Christian-Albrechts-Universit\"at zu Kiel, Ludewig-Meyn-Str.  4, D-24098 Kiel, Germany, E-mail:  christensen@math.uni-kiel.de}, Simon Fischer\thanks{Mathematisches Seminar, Christian-Albrechts-Universit\"at zu Kiel, Ludewig-Meyn-Str.  4, D-24098 Kiel, Germany, E-mail:  fischer@math.uni-kiel.de}}
 \maketitle
 
 \begin{abstract}
  We consider the discrete time stopping problem  
\[ V(t,x) = \sup_{\tau}\e_{(t,x)}[g(\tau, X_\tau)],\]
  where $X$ is a random walk. It is well known that the value function $V$ is in general not smooth on the boundary of the continuation set $\partial C$. We show that under some conditions $V$ is not smooth in the interior of $C$ either. More precisely we show that $V$ is not differentiable in the $x$ component on a dense subset of $C$.
  As an example we consider the Chow-Robbins game. We give evidence that as well $\partial C$ is not smooth and that $C$ is not convex, even if $g(t,\cdot)$ is for every $t$.
 \end{abstract}
 
%  \noindent \textbf{AMS classification:}

\section{Motivation}
Let $X$ be a Markov process and
\begin{equation}\label{eq:vf}
 V(t,x) = \sup_{\tau}\e_{(t,x)}[g(\tau, X_\tau)]
\end{equation}
a stopping problem, where the supremum is taken over a.s. finite stopping times $\tau$. If $X$ is time continuous we usually want to find $V$ on $\RR^{+}\times \RR$ or a sub region. If $X$ is time discrete $V$ is often defined on $\NN\times \RR$ or $\NN\times\ZZ$, but for many problems it seems natural that the process can be started in any point $(t,x)\in \RR^{+}\times \RR$. 
In the continuous setting, $V$ is smooth under some conditions if the smooth fit principle holds, see e.g. \cite{strulovici15}. But even if smooth fit does not hold we can hope to find a solution, using the associated free boundary problem, that will be smooth on the continuation set $C$, see e.g. \cite{peskir06}. In the discrete setting this is however not the case. 
Let $X = (X_n)_{n\in\NN} = (\sum_{i=1}^{n}\xi_{i})_{n\in\NN}$ be a random walk, where the $\xi_i$ are iid. and take discrete values with positive probability. We show in Section \ref{sec:res} that, under some conditions on $g$, the value function \eqref{eq:vf} is not smooth on $C$. More precisely we can show that for every $t$ there is a dense subset of $C\cap\{t\}\times \RR$ on which $V(t,\cdot)$ is not differentiable, see Theorem \ref{th:dif}. As an example we consider the Chow-Robbins game in Section \ref{sec:ex}.
These results lead to the conjecture that the stopping boundary $\partial C$ is not smooth on a dense set either. We will not prove this in general, but give numerical examples in Section \ref{sec:ex}. \\
These results are interesting for different reasons. They show that we can not hope to find a closed form solution for these problems. On the other hand, they give an interesting qualitative characterization of $V$ and $b$ and show that discrete time problems behave quite differently from their time continuous counterparts.

\section{Results}\label{sec:res}
For simplicity we will assume that the gain function $g$ is continuous and partially differentiable in $x$. This condition is not entirely necessary, and could be relaxed.
% but it seems unlikely, that the non-smoothness of $g$ will somehow smooth $V$.
Let $\xi_i$ be iid. real random variables, such that there exist 
$x_1<0<x_2$ with $P(\xi_1=x_1)>0$ and $P(\xi=x_2)>0$ and let $X_n = \sum_{i=1}^{n}\xi_{i}$.
We assume that the stopping problem \eqref{eq:vf} is solved by an a.s. finite stopping time $\tau_{\ast}$, i.e.
\begin{equation}\label{eq:v2}
 V(t,x) = \e[g(t+\tau_{\ast}, x+X_{\tau_{\ast}})].
\end{equation}
 We denote the continuation set with $C$ and the stopping set with $D$. We furthermore assume that the boundary of the continuation set $\partial C$ is the graph of a function $b:\RR^{+}\to \RR$. W.l.o.g. we assume that the stopping set lies above $b$, i.e. $(t,x)\in D$ if $x\geq b(t)$. Before we state our main result, we need the following lemma.

\begin{lemma}\label{lem:conv}
Let $X_n = \sum_{i=1}^{n}\xi_{i}$ and $\xi^{\ast}$ be the maximal upward jump size of $X$ (possibly $\xi^{\ast} = \infty$). If there exists an $\varepsilon>0$ such that $g(t,\cdot)$ is convex on $[b(t)-\varepsilon,b(t)+\xi^{\ast}+\varepsilon]$ for all $t $, then $V(t,\cdot)$ is convex on $(-\infty, b(t)+\xi^{\ast}+\varepsilon]$ for all $t$.
\end{lemma}

\begin{proof}
 Let $(t,x)$ be fixed with $x\leq b(t) + \xi^{\ast}+\frac{\varepsilon}{2}$ and $I:=[x-\frac{\varepsilon}{2}, x+\frac{\varepsilon}{2}]$. For $y,z\in I$ we define 
 \[\tau_{y} :=\inf \{n\mid y+ X_n \geq b(t+n)\}\]
 and
 \[V_y(z):= \e[g(t+\tau_y,z + X_{\tau_y})].\]
 If $y\geq b(t)$, then $\tau_{y} = 0$ and $V_y(z) = g(t,z)$.
  Since $z + X_{\tau_y}\in [b(t+\tau_y)-\varepsilon,b(t+\tau_y)+\xi^{\ast}+\varepsilon]$, $V_y(z)$ is convex on $I$.
  We have $V_y(z) \leq V(t,z)$ for all $y,z\in I$, and $V_y(y) = V(t,y)$, therefore 
  \[V(t,z) = \sup_{y\in I} V_y(z).\]
   As the supremum over convex functions $V(t,\cdot)$ is convex on $I$. Since this construction, particularly $\varepsilon$, does not depend on $x$ and $t$, $V(t,\cdot)$ is convex on $(-\infty, b(t)+\xi^{\ast}+\varepsilon]$ for all $t$.
\end{proof}

\begin{remark}
 For more general processes $X$ a convex gain function $g(t,\cdot)$ does not always yield a convex value function. In \cite{villeneuve07} an example for a diffusion $X$ with linear gain function, is given, that has a non-convex value function, see also \cite{alvarez17}.
\end{remark}

\begin{theorem}\label{th:dif}
 Let $X$, $V$ and $b$ be as defined above. 
 We assume that:
 \begin{itemize}
  \item $V$ does not follow the smooth fit principle, i.e. for every point $(t,x)$ on $\partial C$, the partial derivative $\frac{\partial}{\partial x} V(t.x)$ does not exist,
  \item There exists an $\varepsilon$ such that $g(t,\cdot)$ is convex on $[b(t)-\varepsilon,b(t)+\xi^{\ast}+\varepsilon]$ for every $t$,
  \item $b$ is monotonic,
  \item $b$ is unbounded,
  \item $(b(t+1)-b(t))\to 0$, as $t\to \infty$,
 \end{itemize}
 then
 for every $t$ there is a dense subset of $C\cap\{t\}\times \RR$ on which $V(t,\cdot)$ is not differentiable.\\
\end{theorem}

\begin{proof}
 We split the proof into two parts:
 \begin{enumerate}
  \item We show that if we start in $(t,x)$ and we hit $\partial C$ with positive probability, i.e. there exists $m\in \NN$ such that $P(x+X_m=b(t+m), \tau_{\ast}\geq m)>0$, then $V(t,\cdot)$ is not differentiable in $(t,x)$.
  \item These points are dense in $C\cap\{t\}\times \RR$, for every $t>0$.
 \end{enumerate}

1. Let $(t,x)\in C$ and $m\in \NN$, such that
  \[P(x+X_m=b(t+m),\tau_{\ast}= m)=:p>0,\]
 then
  \begin{align*}
   V(t,x) = & p \cdot g(t+m,b(t+m)) +  \e[g(t+\tau_{\ast},x+X_{\tau_{\ast}})
   \mathbb{I}\{ {\tau_{\ast}< m}\}]\\
    + & \e[g(t+m,x+X_m) \mathbb{I}\{ \tau_{\ast} = m, x + X_m>b(t+m)\}] \\
    +&  \e[V(t+m,x + X_m) \mathbb{I}\{ \tau_{\ast} > m\}],
  \end{align*}
 where $\mathbb{I}\{\}$ denotes the indicator function. We take the right and the left derivative in $x$, should one of these not exist we are already done. We write $g'$ for $\frac{\partial}{\partial x}g$:
 \begin{align*}
 a:=& \lim_{h\downarrow 0}\frac{V_n(x+h) - V_n(x)}{h} = 
 \underbrace{ p \cdot g'(m+n,b(m+n))}_{a_1} \\
 +& \underbrace{ \e[g'(t + \tau_{\ast},x+X_{\tau_{\ast}}) \mathbb{I}\{ {\tau_{\ast}< m}\}]}_{a_2}\\
    + & \underbrace{\e[g'(t+m,x+X_m) \mathbb{I}\{ \tau_{\ast} = m,x+ X_m>b(t+m)\}] }_{a_3}\\
    + & \underbrace{\lim_{h\downarrow 0}\frac{1}{h}\e[(V(t+m,x+X_m+h)- V(t+m,x+X_m)) \mathbb{I}\{ \tau_{\ast} > m\}]}_{a_4},
  \end{align*}
  
 \begin{align*}
  b:=& \lim_{h\downarrow 0}\frac{V(t+m,x) - V(t+m,x-h) }{h} = \\
  & \underbrace{p \cdot  \lim_{h\downarrow 0}\frac{1}{h}
 \big(g(t+m,b(t+m)) - V(t+m,b(t+m)-h) \big)}_{b_1}\\
+&  \underbrace{\e[g'(t+\tau_{\ast},x+X_{\tau_{\ast}}) \mathbb{I}\{ {\tau_{\ast}< m}, x+X_{\tau_{\ast}}>b(t+\tau_{\ast})\}]}_{b_2}\\
+& \underbrace{\lim_{h \downarrow 0}\e[(g(t+\tau_{\ast},x+X_{\tau_{\ast}}) - V(t+\tau_{\ast},x+X_{\tau_{\ast}}-h))  \mathbb{I}\{ {\tau_{\ast}< m}, x+X_{\tau_{\ast}}=b(t+\tau_{\ast})\}]}_{b_2'}\\
   + & \underbrace{\e[g'(t+m,x+X_m) \mathbb{I}\{ \tau_{\ast} = m, x+X_m>b(t+m)\}] }_{b_3}\\
  + & \underbrace{\lim_{h \downarrow 0}\frac{1}{h}\e[(V(t+m,x+X_m) - V(t+m,x+X_m-h))  \mathbb{I}\{ \tau_{\ast} > m\}]}_{b_4}.
 \end{align*}
 We see that $a_1>b_1$ since smooth fit does not hold by assumption,
 $a_2\geq b_2 + b_2'$ since $V(t,\cdot)$ is convex by Lemma \ref{lem:conv},
 $a_3=b_3$ and
 $a_4\geq b_4 $ again since $V(t,\cdot)$ is convex. It follows that $a>b$ and thereby claim 1.\\
 2. Given $(t,x)\in C$ and $\varepsilon >0$, we show that there is an 
 $x'\in (x-\varepsilon, x+\varepsilon)$, such that $(t,x')$ fulfills the assumptions of proof part 1, i.e.
 there exists a $m\in \NN$, such that $P(x' + S_m=b(t+m)', \tau_{\ast} = m)>0$.  
 We distinguish two cases.\\
 Case 1. The upward jump sizes and the downward jump sizes with positive probability have a common multiple.
 Then there exists a $m^{\ast}\in\NN$ with $P(S_{m^{\ast}}=0)>0$. 
 There exists a series increments $\lambda_1, \dots,\lambda_{m^{\ast}}$ such that 
 $\sum^{m^{\ast}}_{i=1}\lambda_i = 0$ and $P(\xi_1=\lambda_1, \dots,\xi_{m^{\ast}} = \lambda_{m^{\ast}})>0$. Changing the order of the increments does not change the probability, i.e. we have a series of increments $\Lambda_{m^{\ast}} = (\lambda_{k_1}, \dots,\lambda_{k_{m^{\ast}}})$ with the above properties and $\lambda_{k_i}\leq \lambda_{k_{i+1}}$ for all $i\leq m^{\ast}-1$.\\
We choose $N$ such that $b(s+m^{\ast})-b(s)<\varepsilon$ for all $s\geq N$. We choose $m$ such that $t+m\geq N$ and $y\geq b(t+m)$ with $P(x+S_m=y)> 0$. 
Again we find a series of increments $\Lambda_{y} = (\lambda_{1}^{y}, \dots,\lambda_{m}^{y})$, with $x+\sum^{m}_{i=1}\lambda_i^{y} = y$, $P(\xi_1=\lambda_1^{1}, \dots,\xi_{m} = \lambda_{m}^{y})>0$ and $\lambda_{i}^{y}\leq \lambda_{i+1}^{y}$. An illustration of the construction in this part is shown in Figure \ref{fig:path}.\\
Let $k^{\ast} = \min \{k \mid b(t+m+k m^{\ast})\geq y\}$, then 
$b(t+m+k^{\ast} m^{\ast})-y=:\varepsilon'\leq \varepsilon$. Now $x':=x+\varepsilon'$ is our candidate starting point, and we need to show that there is a path from $(t,x')$ to $(t+m+k m^{\ast},b(t+m+k m^{\ast}))$, that has positive probability and lies in $C$. We take the path with the increments 
$\underbrace{\Lambda_{m_\ast}\dots \Lambda_{m_\ast}}_{ k^{\ast} \text{ times}}\Lambda_y$, this clearly has positive probability. The first $k^{\ast}m^{\ast}$ steps clearly lie in $C$, but the last part might not, since we can not assume $C$ to be convex. If it does not lie in $C$ we start the same procedure with $(t+km,k(y-x)+x)$ instead of $(t+m,y)$ and since $b(t)$ grows slower with increasing $t$, but the jump sizes in $\Lambda_y$ 
does not change, the claim will hold for some $k$.\\ 

   \begin{figure}[h]
	\centering
  \includegraphics[width=0.9\textwidth]{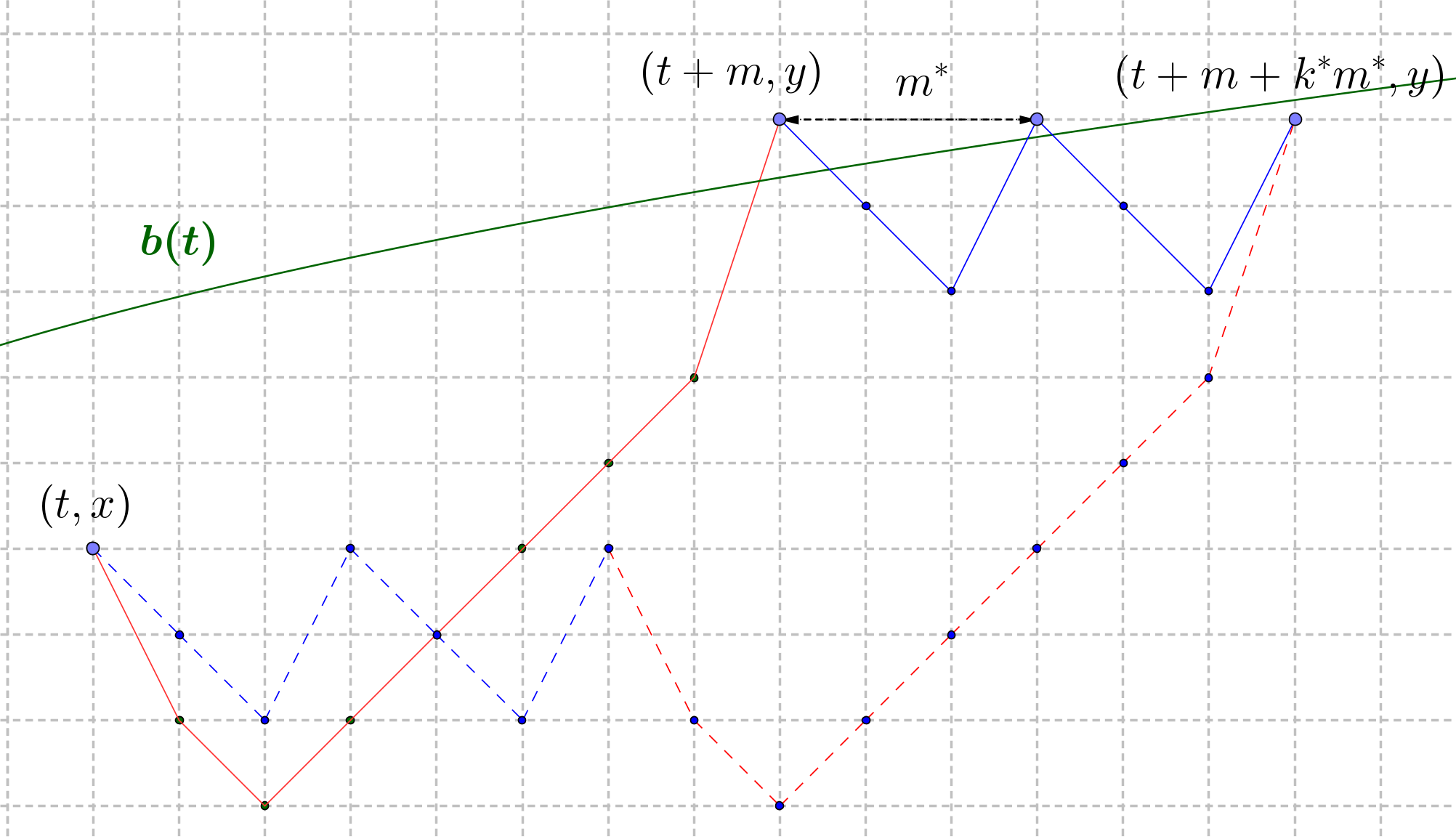}
	\caption{Possible paths from $(t,x)$ to $(t+m+k^{\ast}m^{\ast},y)$}
	\label{fig:path}
\end{figure}

Case 2. The upward jump sizes and the downward jump sizes with positive probability have no common multiple.\\
We find a $m^{\ast}$, a $-\frac{\varepsilon}{2}<\tilde x \leq 0$ and $N$ such that $P(S_{m^{\ast}}=\tilde x)>0$ and $b(s+m^{\ast})-b(s)<\frac{\varepsilon}{2}$ for all $s\geq N$. Now the rest follows analogous to case 1.
 \end{proof}
 
\begin{remark}
 The first part of the proof stays true for almost any discrete stopping problem, that does not have smooth fit. Most of the assumptions were needed to prove that the points lie dense.
\end{remark}

\begin{remark}
 $X$ does not need to be a random walk. The proof works the same way, if $X$ is a discrete time Markov process, such that there exist 
$x_1<0<x_2$ with $P(X_{n+1}-X_n=x_1\mid \mathcal{F}_n)>0$ and $P(X_{n+1}-X_n=x_2\mid \mathcal{F}_n)>0$ a.s. for all $n\in \NN$, where $\mathcal{F}_n$ is the natural filtration.
 
\end{remark}

\begin{remark}
 The non smoothness of $V$ suggests that in the setting of Theorem \ref{th:dif} the stopping boundary $b(t)$ is not smooth either. In points $(t,b(t))$ on the boundary from which we could jump exactly onto the boundary again we would expect that if existent
 \[\lim_{h\uparrow 0} \frac{1}{h}(b(t+h)-b(t))<\lim_{h\downarrow 0} \frac{1}{h}(b(t+h)-b(t)).\]
 These points lie dense on $\partial C$. This means in particular that the continuation set $C$ is not convex. We will not prove this conjecture in general, but study examples in the next section.
\end{remark}

\section{Examples}\label{sec:ex}
The assumptions in Theorem \ref{th:dif} may seem a bit arbitrary at first sight, but they hold for many stopping problems. The slow growth condition for $b$, $(b(n+1)-b(n))\to 0$, follows from the law of the iterated logarithm if the $\xi_i$ have second moments, the monotonicity can often be deduced directly from $g$, e.g. if $g(\cdot,x)$ is non increasing on $D$.
That $V$ has no smooth fit can be assumed in most cases, and can be shown as in Example \ref{ex:1}. \\
As an explicit example we want to mention the Chow-Robbins game. 

\begin{example}\label{ex:1}
 Let $\xi_1, \xi_2 \dots$ be iid. random variables with 
$P(\xi_i =-1) =P(\xi_i =1) = \frac{1}{2} $,
$X_n = \sum_{i=1}^{n}\xi_i$ and $g(t,x)=\frac{x}{t}$.
The stopping problem
\begin{equation}\label{eq:cr}
V(t,x)=\sup_{\tau  }\e \left [ \frac{x+ X_\tau}{t+\tau}\right ].
\end{equation}
is called Chow-Robbins game. It was introduced in 1965 in \cite{chow1965}. In \cite{snproblem} the authors recently showed how to calculate a good approximation for \eqref{eq:cr}. The stopping problem \eqref{eq:cr} fulfills the conditions of Theorem \ref{th:dif}. $g(t,\cdot) = \frac{\cdot}{t}$ is convex for all $t$, hence so is $V(t,\cdot)$. To see that smooth fit does not hold we calculate the left and right derivative, we write $g'$ for $\frac{\partial}{\partial x}g$:
\begin{align*}
 & \lim_{h\downarrow 0}\frac{V(t,b(t)+h) - V(t,b(t))}{h} = g'(t,x) = \frac{1}{t}
\end{align*}

\begin{align*}
 & \lim_{h\uparrow 0}\frac{V(t,b(t)+h) - V(t,b(t))}{h} \\
 =& \underbrace{\lim_{h\uparrow 0}\frac{1}{2h}\Big( g\big(t+1,b(t)+h+1\big) - g\big(t+1,b(t)+1\big)}_{=g'(t+1,b(t)+1)} \\
 &+ \underbrace{V\big(t+1,b(t)+h-1 \big) - V\big(t+1,b(t)-1\big) \Big)}_{\leq g'(t+1,b(t)-1) \text{ since $V$ is convex}}\\
 \leq & \frac{1}{2(t+1)} + \frac{1}{2(t+1)} =\frac{1}{t+1}<\frac{1}{t}.
\end{align*}
For the properties of $b(t)$ see \cite{chow1965, snproblem, lai2006}.
 With Theorem \ref{th:dif} we see that the value function $V$ is not differentiable on a dense subset of the continuation set C. 
Figure \ref{fig:1} shows $V$ for the fixed time $t=1$. Some non smooth points can be seen in the plot:
\begin{itemize}
 \item $x_0 = 0.46$ is the smallest value of $x$ for which it is optimal to stop. We see that $V$ does not follow the smooth fit principle.
 \item $x_1 = -0.22$ is the smallest value for which $(2,x+1)$ is in the stopping set $D$.
  \item $x_2 = -0.97$ is the smallest value for which $(3,x+2)\in D$. 
  \end{itemize}
In Figure \ref{fig:2} $V$ is given for $t=5$.
  
  \begin{figure}[h]
	\centering
  \includegraphics[width=0.8\textwidth]{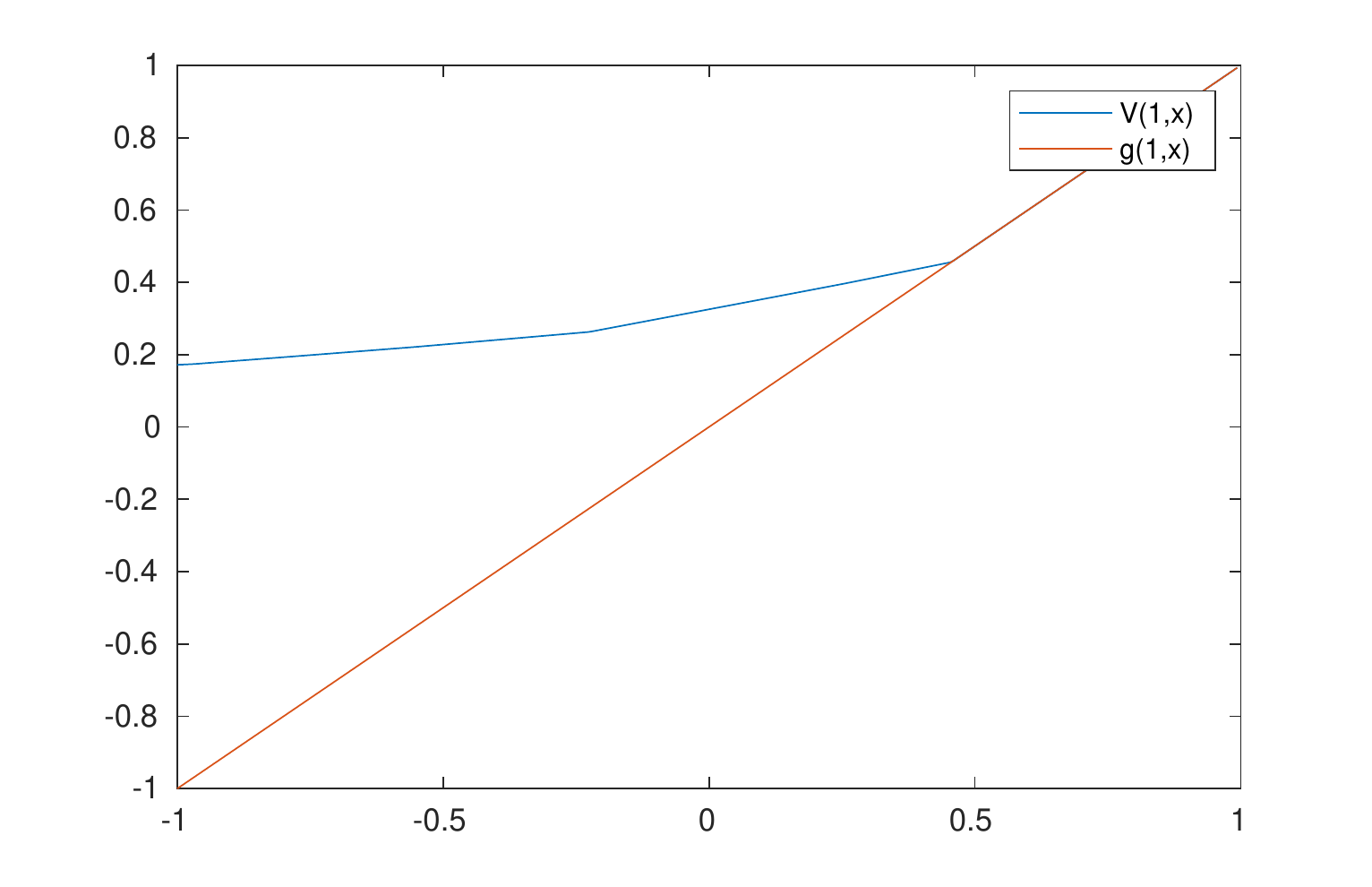}
	\caption{The value function of Example \ref{ex:1} $V(1,\cdot)$ (blue) and the gain function $g(1,\cdot)$ (orange). Some non smooth points can be seen. }
	\label{fig:1}
  \end{figure}

  \begin{figure}[h]
	\centering
  \includegraphics[width=0.8\textwidth]{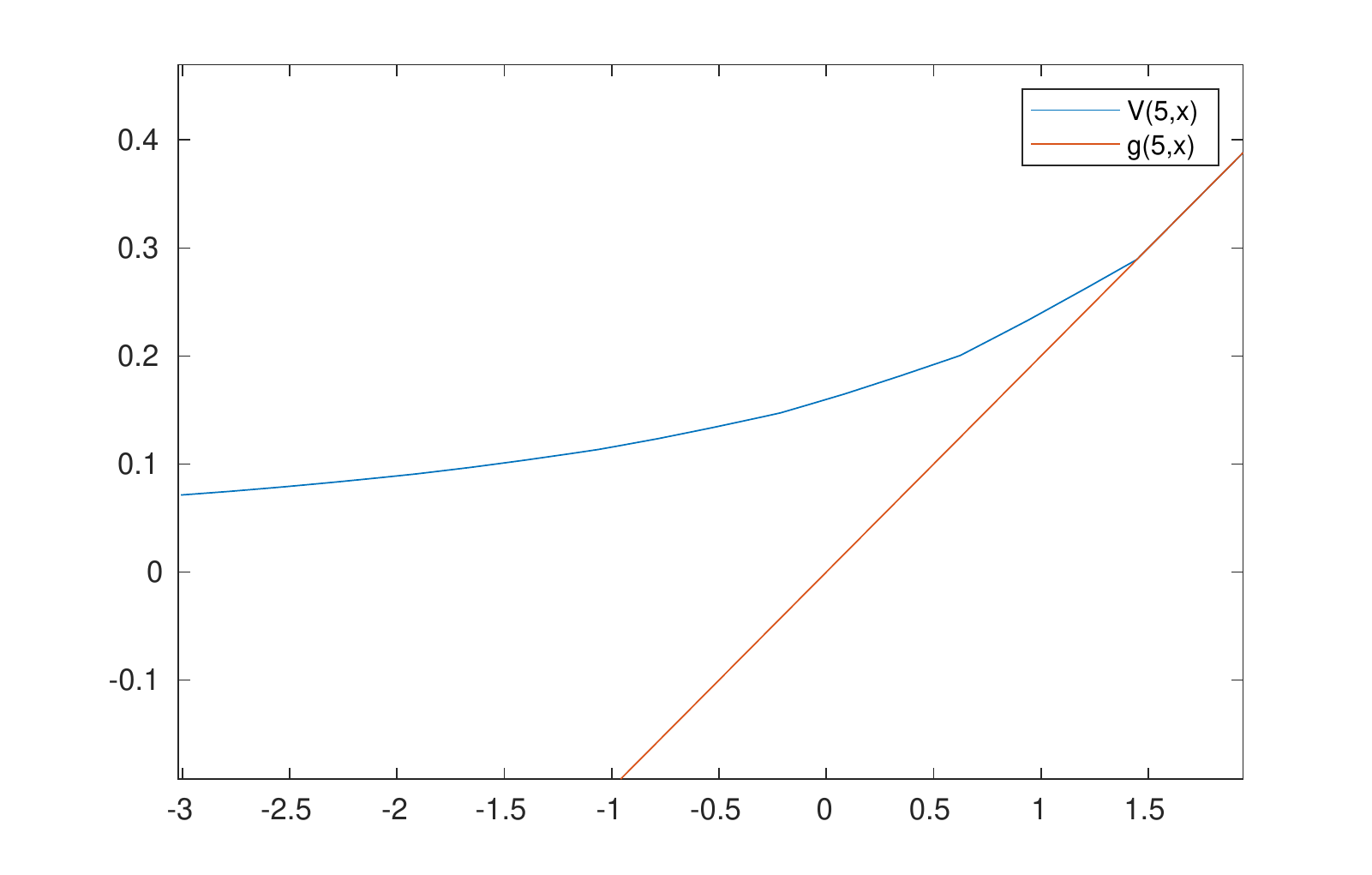}
	\caption{The value function of Example \ref{ex:1} $V(5,\cdot)$ (blue) and the gain function $g(5,\cdot)$ (orange).  }
	\label{fig:2}
\end{figure}

No one found a closed form for $V$ so far. The fact that $V$ is not differentiable on a dense set shows, that it is highly unlikely that this is even possible.
\end{example}

\begin{example}[$C$ is not convex]\label{ex:2}
The continuation set of the previous example is not convex and the stopping boundary most likely not differentiable in $t$, see Figure \ref{fig:nconvS}. We change the setting slightly, in order to make the effect stronger and more visible.\\

  \begin{figure}[h]
	\centering
  \includegraphics[width=0.9\textwidth]{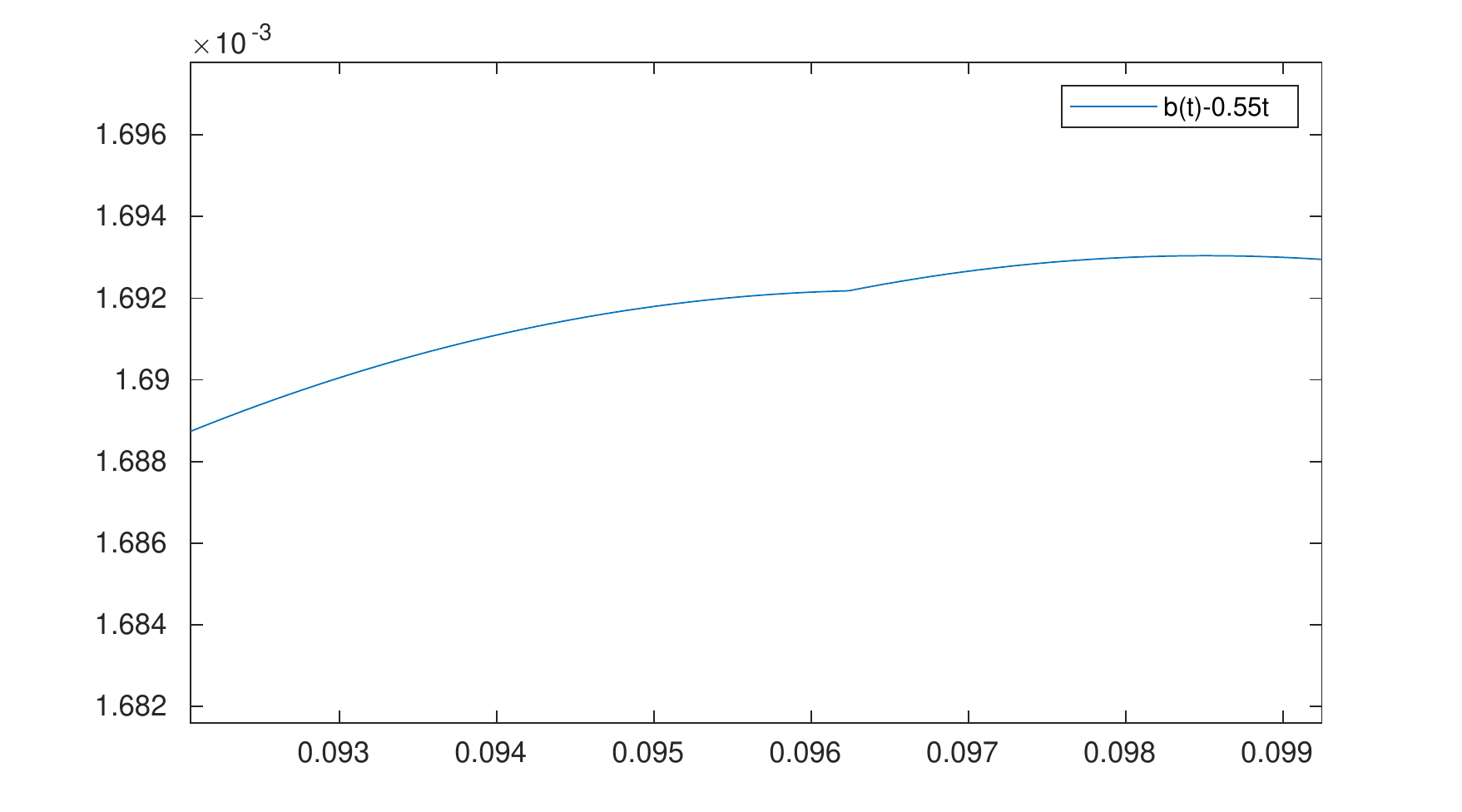}
	\caption{The tilted stopping boundary $b(t)-0.55t$ of the Chow-Robbins game in Example \ref{ex:1}}
	\label{fig:nconvS}
\end{figure}

Let $\xi_1, \xi_2 \dots$ be iid. random variables with 
$P(\xi_i =-1.5) = \frac{24}{85}$, $P(\xi_i =0.2) =\frac{25}{68} $, $ P(\xi_i+1) =\frac{7}{20} $,
$X_n = \sum_{i=1}^{n}\xi_i$, $g(t,x)=\frac{x}{t}$ and
$V(t,x)=\sup_{\tau  }\e \left [ \frac{x+ X_\tau}{t+\tau}\right ].$
The $\xi_i$ are centered, have unit variance and the value function $V$ has the upper bound given in \cite{snproblem}. We numerically calculate an estimation of $V$ and with that the stopping boundary $b(t)$, the absolute error of our calculation is approximately $
10^{-6}$. In Figure \ref{fig:nconv2} we see the stopping boundary $b(t)$. It looks smooth and concave, only if we zoom in and tilt it for better visibility, we see that this conception is misleading.\\
The point $(3.697,1.089)$ lies on the boundary $\partial C$ and $(3.697+1,1.089+0.2)$ again lies on $\partial C$, hence we expect $b(t)$ to be non smooth in $t=3.697$. In Figure \ref{fig:nconv} we see a plot of $b(t)-0.2085t$. The effect is small, but we can see clearly, that $C$ is not convex. If we zoom in further, we can see more non-smooth points, see Figure \ref{fig:nconv3}.

  \begin{figure}[h]
	\centering
  \includegraphics[width=0.9\textwidth]{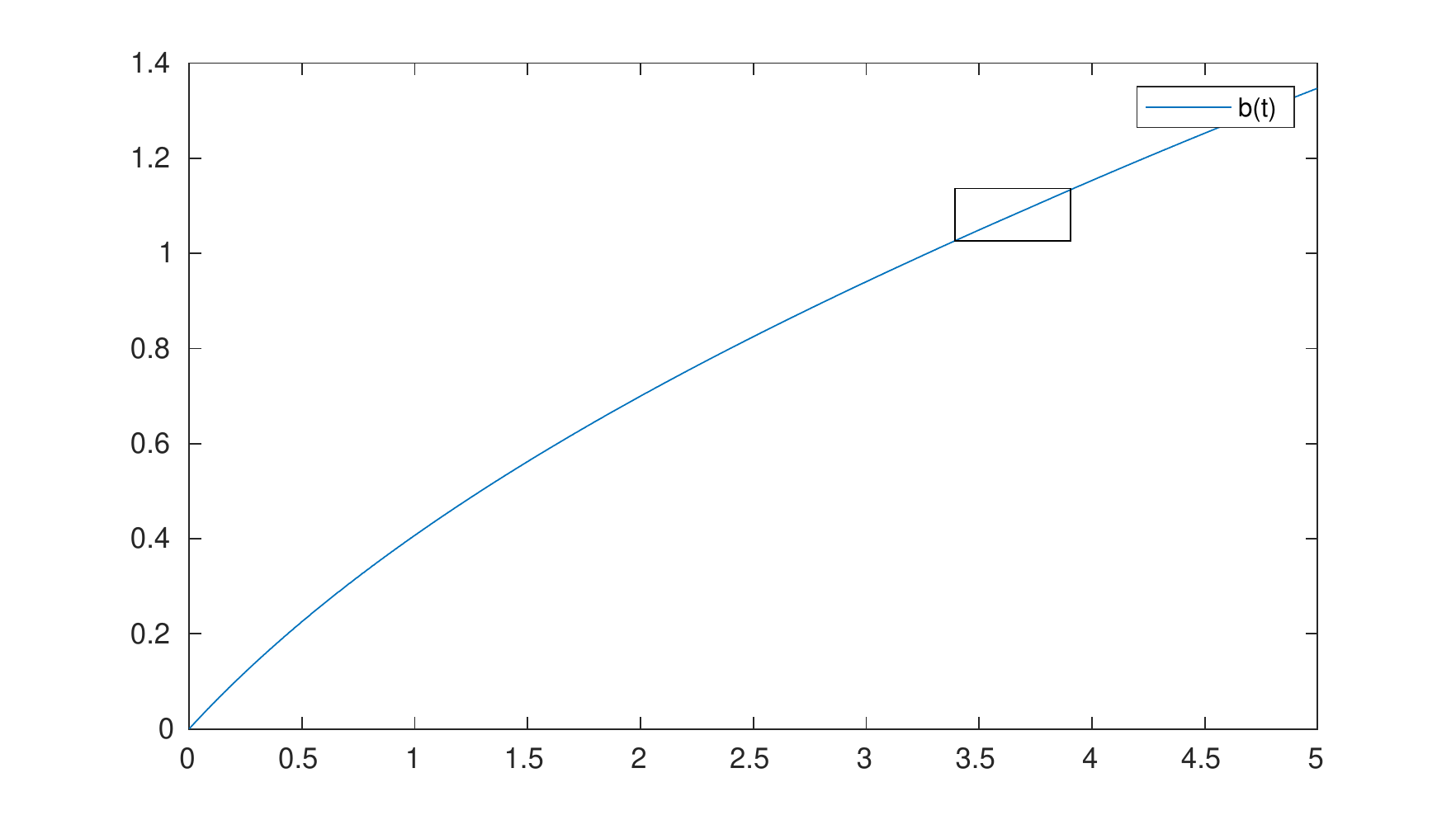}
	\caption{The stopping boundary $b(t)$ of the stopping problem in Example \ref{ex:2}. The boxed part is shown in Figure \ref{fig:nconv}.}
	\label{fig:nconv2}
\end{figure}

  \begin{figure}[h]
	\centering
  \includegraphics[width=0.9\textwidth]{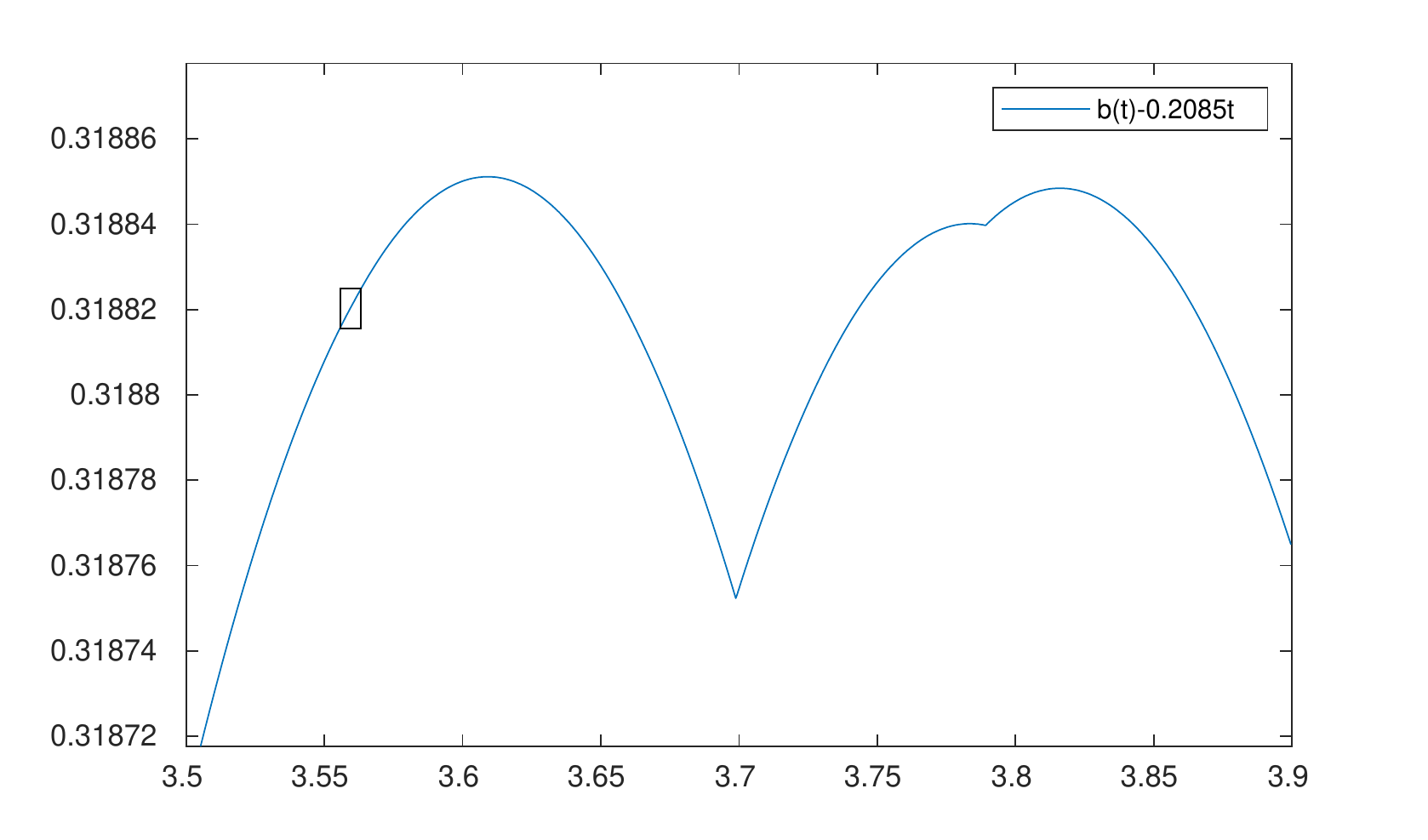}
	\caption{The tilted stopping boundary $b(t)-0.2085t$ of the stopping problem in Example \ref{ex:2}. The boxed part is shown in Figure \ref{fig:nconv3}.}
	\label{fig:nconv}
\end{figure}

  \begin{figure}[h]
	\centering
  \includegraphics[width=0.9\textwidth]{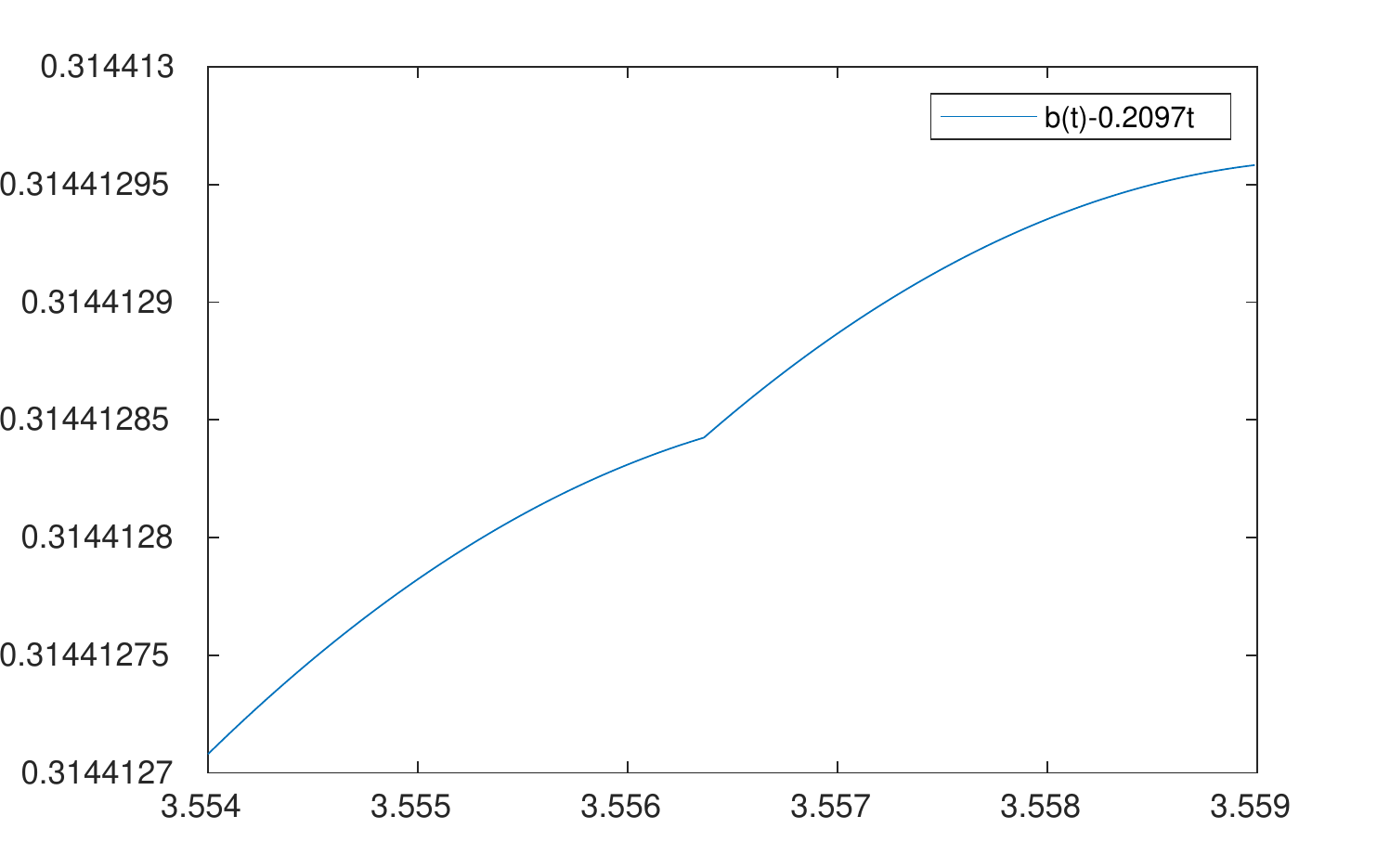}
	\caption{The tilted stopping boundary $b(t)-0.2097t$ of the stopping problem in Example \ref{ex:2}}
	\label{fig:nconv3}
\end{figure}

\end{example}

To conclude this section we give some examples, why some assumptions for Theorem \ref{th:dif} are necessary.

\begin{example}
 Smooth fit holds: Let $g$ be smooth with $g(t,x) = 0$ for $x\geq \sqrt{t}$, $g(t,x) < 0$ for $x< \sqrt{t}$ and $X$ a random walk. Then the problem fulfills all assumptions but the no smooth fit one and $V\equiv 0$, hence $V$ is smooth everywhere.
\end{example}

\begin{example}\label{ex:4}
 $b$ is bounded: Let $X$ be a Bernoulli random walk $P(\xi_i =-1) =P(\xi_i =1) = \frac{1}{2} $ and 
 $g(t,x) = \begin{cases}
-x^{2} & x \leq 0 \\
-\min\{\lceil x \rceil-x,x-\lfloor x\rfloor\}^2 \ & x > 0.
\end{cases}$
 Then $b(t) = -\frac{1}{2}$ and
 $V(t,x)=-\min\{\lceil x \rceil-x,x-\lfloor x\rfloor\}^2$ is differentiable if $\lceil x \rceil-x\neq \frac{1}{2}$, see Figure \ref{fig:4}.

   \begin{figure}[h]
	\centering
  \includegraphics[width=0.7\textwidth]{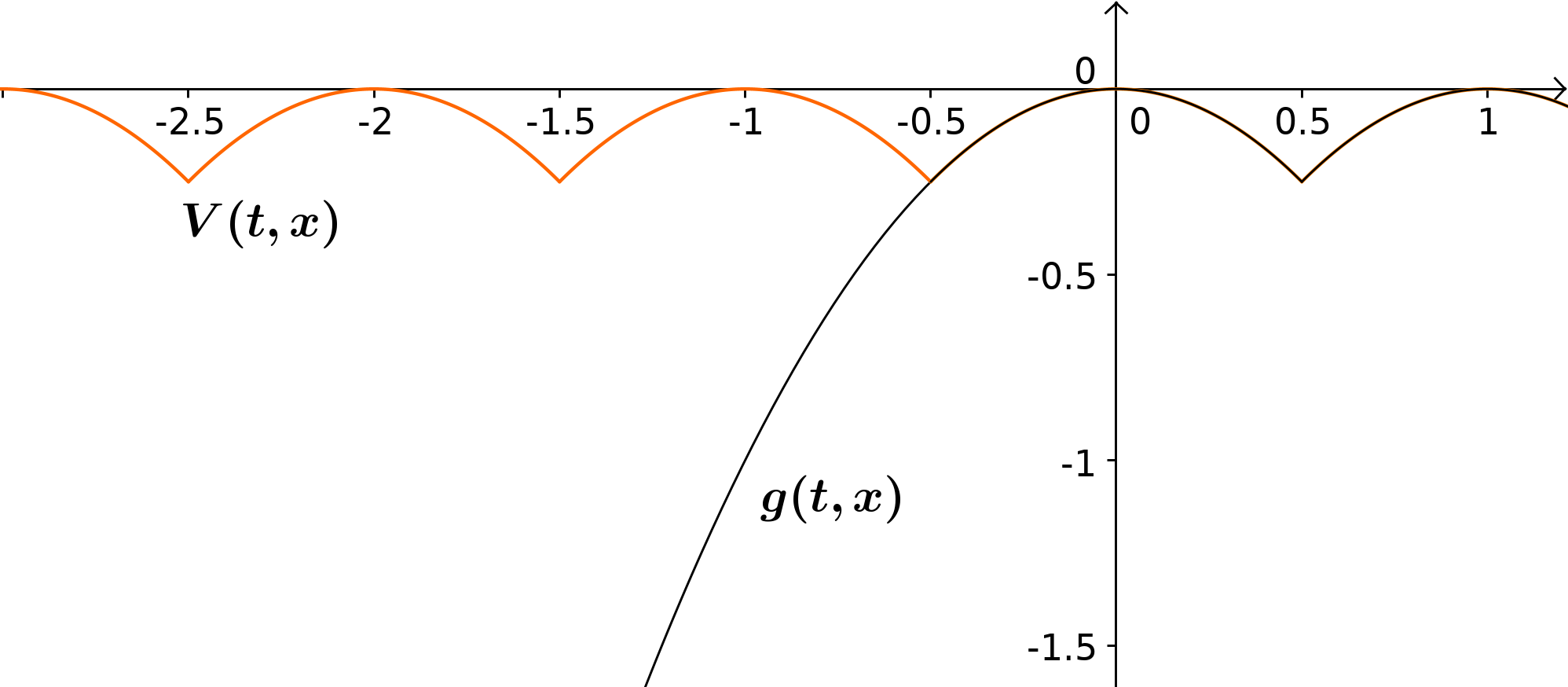}
	\caption{The value function of Example \ref{ex:4} }
	\label{fig:4}
\end{figure}

\end{example}

\section{Conclusion}
We have shown that the value functions of the analyzed stopping problems is not differentiable on a dense subset of $C$. We have shown the non-smoothness in the $x$ component, but the value function will not be differentiable in $t$ in the \textit{critical points} either. Furthermore, it seems likely that the stopping boundary $b(t)$ is not smooth on a dense set as well.
This shows that it is highly unlikely to find a closed form for $V$ or $b$.
Although we can use discrete time stopping problems to approximate continuous time problems and vice versa, their solutions may have different analytical properties. This shows, that we need to be careful with assumptions about the solutions of discrete time problems, because our intuition might be misleading.
We restricted our proof to specific cases, but the described phenomena seem to be typical for discrete time stopping problems. \\
Examples of functions that are continuous but not differentiable on a dense subset can be found in \cite{klatte02}.

 \clearpage

   \bibliographystyle{alpha}
 \bibliography{Vnichtdiffbar}
\end{document}